\documentclass[twocolumn,noversion,nonote,squeezed]{cdcarticle}
\newcommand{\map}[3]{#1: #2 \rightarrow #3}
\newcommand{\setdef}[2]{\{#1 \;|\; #2\}}
\usepackage{colonequals}
\newcommand{\jdefine}{\colonequals}

\title{Automated Lyapunov Analysis of Primal-Dual Optimization Algorithms: An Interpolation Approach}
\author{Bryan Van Scoy\thanks{B. Van Scoy is with the Department of Electrical and Computer Engineering at Miami University, Oxford, OH 45056, USA.
\texttt{bvanscoy@miamioh.edu}}\and John W. Simpson-Porco\thanks{J. W. Simpson-Porco is with the Department of Electrical and Computer Engineering at the University of Toronto, Toronto, ON, M5S 3G4, Canada. \texttt{jwsimpson@ece.utoronto.ca}}\and Laurent Lessard\thanks{L. Lessard is with the Department of Mechanical and Industrial Engineering at Northeastern University, Boston, MA 02115,
USA. \texttt{l.lessard@northeastern.edu}\\[2pt]
This material is based upon work supported by the National Science Foundation under Grant No. 2136945 and 2139482.}}
\note{IEEE Conference on Decision and Control, 2023}

\usepackage{graphicx}
\usepackage{math}
\usepackage{cite}
\usepackage{color}
\usepackage{bm}
\usepackage[hidelinks]{hyperref}
\usepackage[font=small,labelfont=bf,margin=1em]{caption}
\let\oldbibliography\thebibliography
\renewcommand{\thebibliography}[1]{%
  \oldbibliography{#1}%
  \setlength{\itemsep}{1pt}%
}

\def\F{\mathcal{F}}

\def\M{\mathcal{M}}
\def\A{\mathcal{A}}

\def\GG{\mathbf{G}}
\def\AA{\mathbf{A}}
\def\BB{\mathbf{B}}
\def\CC{\mathbf{C}}
\def\DD{\mathbf{D}}
\def\nn{\mathbf{n}}
\def\YF{\mathbf{Y_f}}
\def\UF{\mathbf{U_f}}
\def\YA{\mathbf{Y_A}}
\def\UA{\mathbf{U_A}}
\def\real{\mathbb{R}}

\newcommand{\sbmat}[1]{\left[\begin{smallmatrix}#1\end{smallmatrix}\right]}

\theoremstyle{definition}
\newtheorem{theorem}{Theorem}
\newtheorem{proposition}{Proposition}

\begin{document}

\maketitle

\begin{abstract}
Primal-dual algorithms are frequently used for iteratively solving large-scale convex optimization problems. The analysis of such algorithms is usually done on a case-by-case basis, and the resulting guaranteed rates of convergence can be conservative. Here we consider a class of first-order algorithms for linearly constrained convex optimization problems, and provide a linear matrix inequality (LMI) analysis framework for certifying worst-case exponential convergence rates. Our approach builds on recent results for interpolation of convex functions and linear operators, and our LMI directly constructs a Lyapunov function certifying the guaranteed convergence rate. By comparing to rates established in the literature, we show that our approach can certify significantly faster convergence for this family of algorithms.
\end{abstract}

%%%%%%%%%%%%%%%%%%%%%%%%%%%%%%%%%%%%%%%%%%%%%%%%%%%%%%%%%%%%%%%%%%%%%%%%%%
\section{Introduction}

Primal-dual (or saddle-point) optimization methods have a rich history, dating back to the earliest days of mathematical programming \cite{TK:56,KA-LH-HU:58}. The core idea --- that of sequentially or simultaneously updating both primal and dual variables --- is now widely used in algorithms for solving constrained optimization problems, including in interior-point methods \cite{numerical-optimization}, the method of multipliers, and in distributed optimization methods such as ADMM \cite{admm}. 

There is significant overlap with the literature on \emph{operator splitting}, as finding a point satisfying the KKT conditions of a constrained optimization problem can be cast as the problem of finding a zero of a sum of monotone operators; see \cite{NK-JCP:15,PLC-JCP:21,LC-DK-AC-AK:23} for extensive overviews of this perspective. Convergence proofs in this literature generally rely on the construction of bespoke pre-conditioners, followed by applications of convergence results for known fixed-point algorithms (e.g., Krasnoselskii-Mann iterations), or by clever direct construction of Lyapunov-like functions. While we do not herein study operator splitting methods in generality, part of our goal is to establish some preliminary foundations for more systematic and automated analyses of such algorithms.

%\cite{PLC-JCP:21,FJ-ZZ-HE:22}

Primal-dual methods have attracted attention from controls researchers, particularly in the continuous-time setting where Lyapunov analysis techniques can be applied with relative ease; see \cite{DF-FP:10,AC-EM-JC:15,JWSP:16f,AC-EM-SL-JC:17,JWSP-BKP-NM-FD:18b,NKD-SZK-MRJ:19,DD-MRJ:19, DD-MRJ:20, XC-NL:20}. 
% removed citations: HM-MR-MRJ:20
This has led to various control applications of the algorithms, such as in energy systems \cite{NL-CL-ZC-SHL:13,AC-JC:14,TS-CDP-AVDS:15,JWSP-BKP-NM-FD:16c,AB-ED:19}. However, only \cite{DD-MRJ:20} directly addresses the issue of the \emph{rate} of exponential convergence in discrete time, and the analysis presented is focused on a particular algorithm obtained via Euler discretization from the continuous-time version. 

By interpreting optimization algorithms as dynamical systems, specifically as robust controllers, integral quadratic constraints (IQCs) have been used to find tight bounds on convergence rates \cite{lessard,michalowsky2021robust}. Alternatively, one can directly generate a set of valid inequalities relating inputs and outputs of the objective function, and solve a \emph{meta}-optimization problem that searches for tight worst-case guarantees. This was applied in a \emph{finite-horizon} setting in the so-called PEP formulation \cite{interpolation}, and also in an asymptotic setting \cite{lifting,lessard2022analysis,hu2017dissipativity} to directly search for Lyapunov functions that certify a given convergence rate.

To the best of our knowledge, the aforementioned approaches have not previously been applied to analyze primal-dual algorithms for linearly constrained convex optimization. The closest works we found examined over-relaxed ADMM \cite{nishihara2015general}, or alternating gradient methods for bilinear games \cite{pmlr-v151-zhang22e} or smooth monotone games \cite{zhang2021unified}.

\paragraph{Contributions:} We consider a family of first-order primal-dual algorithms for solving linearly constrained convex optimization problems of the form
\begin{equation}\label{Eq:Opt}
    \minimize_{x \in \real^n} \,\, f(x) \quad \text{subject to} \quad Ax = b.
\end{equation}
Our main contribution is an automated framework for computing worst-case convergence rates of the algorithm over a class of problem data. We consider smooth and strongly convex $f$ and matrices $A$ with known bounds on the singular values. We show numerically that our analysis improves on known results. We also develop the set of multipliers for the class of smooth strongly convex functions and the class of linear functions with eigenvalues in a closed interval which may be of independent interest.

% \emph{Outline:} We describe the primal-dual algorithm and known bounds from the literature in Section \ref{Sec:PrimalDual}. We then develop our analysis in Section \ref{Sec:analysis}, present our results in Section \ref{Sec:results}, and conclude in Section \ref{Sec:conclusion}.

%----------------------------------------------------------------------
\section{Primal-dual iterations for linearly constrained convex optimization}
\label{Sec:PrimalDual}

Consider the optimization problem~\eqref{Eq:Opt},
%
% \begin{equation}\label{Eq:Opt}
%     \minimize_{x \in \real^n} \,\, f(x) \quad \text{subject to} \quad Ax = b,
% \end{equation}
where the goal is to minimize the objective function $\map{f}{\real^n}{\real}$ over the affine constraint set $\mathsf{C} \jdefine \setdef{x\in\real^n}{Ax = b}$ with $A \in \real^{r \times n}$ and $b \in \real^r$. Throughout this work, we make the following assumptions, which (among other things) ensure that the problem is feasible for any $b \in \real^r$ and possess a unique optimal solution $x_\star \in \real^n$. 

\begin{enumerate}
\item[(i)] For known constants $0 < m \leq L < \infty$, the objective function $f$ is $m$-strongly convex, continuously differentiable, and its gradient $\map{\nabla f}{\real^n}{\real^n}$ is globally Lipschitz continuous with Lipschitz constant $L$; we denote the set of all such functions by $\F(m,L)$, and we denote the condition number as $\kappa(f) = L/m$.
\item[(ii)] For known constants $0 < \underline{\sigma} \leq \overline{\sigma} < \infty$, the singular values $\sigma_i(A)$ of the constraint matrix $A \in \real^{r \times n}$ satisfy $\underline{\sigma} \leq \sigma_i(A) \leq \overline{\sigma}$ for all $i \in \{1,\ldots,r\}$; we denote the set of all such matrices by $\A(\underline{\sigma},\overline{\sigma})$, and we denote the condition number as $\kappa(A) = \overline{\sigma}/\underline{\sigma}$.
\end{enumerate}

Note that (ii) implies that $A$ has full row rank, and thus the constraints $Ax = b$ are linearly independent. 

Perhaps the most immediate iterative approach for computing the optimal solution of \eqref{Eq:Opt} would be projected gradient descent
\[
x_{k+1} = \mathrm{Proj}_{\mathsf{C}}(x_k - \alpha \nabla f(x_k)), \qquad k = 0,1,\ldots
\]
with step size $\alpha > 0$. In many large-scale applications however, computing this projection is too computationally expensive, and one encounters similar computational bottlenecks if gradient ascent is applied to the dual problem of \eqref{Eq:Opt}. Instead, iterative methods are sought which rely only on (a small number of) evaluations of $\nabla f$, $A$, and $A^{\tp}$ at each iteration \cite{NK-JCP:15,PLC-JCP:21}. 

Such methods can be developed through Lagrange relaxation of the equality constraint in \eqref{Eq:Opt}. For $\mu \geq 0$ define the \emph{augmented Lagrangian}
\[
L_{\mu}(x,\lambda) = f(x) + \lambda^{\tp}(Ax-b) + \tfrac{\mu}{2}\|Ax-b\|_2^2,
\]
with $\lambda \in \real^r$ the dual variable and $\mu$ the augmentation parameter. Under the present assumptions, strong duality holds, and $(x_{\star},\lambda_{\star})$ is primal-dual optimal for \eqref{Eq:Opt} if and only if it is a \emph{saddle point} of $L_{\mu}$. As an iterative method to determine a saddle point, one begins with any initial condition $(x_0,\lambda_0)$ and performs the gradient decent-ascent iterations
\begin{subequations}
\begin{align}\label{Eq:GDA}
x_{k+1} &= x_{k} - \alpha_{x} \nabla_{x}L_{\mu}(x_{k},\lambda_{k}) \notag\\
&= x_{k} - \alpha_{x}[\nabla f(x_{k})  + A^{\tp}\lambda_{k} + \mu A^{\tp}(Ax_k-b)]\\
\lambda_{k+1} &= \lambda_{k} + \alpha_{\lambda} \nabla_{\lambda}L_{\mu}(x_{k},\lambda_{k})\notag\\
&= \lambda_{k} + \alpha_{\lambda}(Ax_k-b)
\end{align}
\end{subequations}
where $\alpha_x, \alpha_{\lambda} > 0$ are step sizes. For obvious reasons, such algorithms are termed \emph{primal-dual} algorithms. In this work, we consider a slightly more general variation on \eqref{Eq:GDA}, given by
\begin{subequations}
\begin{align}\label{Eq:AHU-Extended}
x_{k+1} &= x_{k} - \alpha_{x}[\nabla f(x_{k})  + A^{\tp}\lambda_{k} + \mu A^{\tp}(Ax_k-b)]\\
\tilde{x}_{k} &= x_k + \gamma (x_{k+1}-x_k)\\
\lambda_{k+1} &= \lambda_{k} + \alpha_{\lambda}\left(A\tilde{x}_{k}-b\right),
\end{align}
\end{subequations}
where $\gamma \in [0,2]$ is an \emph{extrapolation} parameter. Various algorithms are contained as special cases of \eqref{Eq:AHU-Extended} and will serve as points of comparison. Our broad goal is to quantify the \emph{worst-case} asymptotic geometric convergence rates achieved by some selected iterative primal-dual algorithms over all possible instances of problem data $f \in \mathcal{F}(m,L)$ and $A \in \mathcal{A}(\underline{\sigma},\overline{\sigma})$.

\subsection{Literature on known rates}
\label{Sec:literature}

For sufficiently small step sizes, \eqref{Eq:AHU-Extended} converges  exponentially to the unique saddle point $(x_{\star},\lambda_{\star})$ of $L_{\mu}$. A significantly more challenging question is to provide non-conservative estimates of the worst-case asymptotic geometric convergence rate for the method over the class of problem data defined by $(\mathcal{F},\mathcal{A})$. 

We have found two relatively clear comparison points. First, \cite{SSD-WH:19} considers \eqref{Eq:AHU-Extended} with $\mu = \gamma = 0$. Translating the notation\footnote{In the notation of \cite{SSD-WH:19}, our set-up corresponds to the case where $f = 0$; our ``$A$'' is their ``$A^{\sf T}$'' and our ``$f$'' is their ``$-g$''.}, they use a Lyapunov function of the form
\[
    V(x_k,\lambda_k) = \|x_{k} - \nabla f^*(-A^{\tp}\lambda_k)\|_2 + c\,\|\lambda_k - \lambda_{\star}\|_2,
\]
where $f^*(z) = \sup_{x \in \real^n} x^\tp z - f(x)$ is the convex conjugate of $f$ \cite{convex-analysis} and $c \defeq 2\frac{L}{m^2}\frac{\overline{\sigma}^3}{\underline{\sigma}^2}$. Using stepsizes
\begin{equation}\label{stepsizes1}
    \alpha_{x} = \frac{2}{m+L} \quad\text{and}\quad
    \alpha_{\lambda} = \frac{m}{(m+L)\bigl(\frac{\overline{\sigma}^2}{m}+c\overline{\sigma}\bigr)},
\end{equation}
the decrease condition $V_{k+1} \leq \rho\,V_k$ holds with
\begin{equation}\label{bound1}
    \rho = 1 - \frac{1}{12\,\kappa(f)^3\,\kappa(A)^4}.
\end{equation}
Second, the authors of \cite{SAA-AHS:20} consider \eqref{Eq:AHU-Extended} with $\gamma = 1$, and for notational simplicity we consider here $\mu = 0$. Using the quadratic Lyapunov function
\[
V(x_k,\lambda_k) = (1-\alpha_x\alpha_{\lambda}\overline{\sigma}^2)\|x_k-x_{\star}\|_2^2 + \|\lambda_k-\lambda_{\star}\|_2^2
\]
and the step size conditions $\alpha_x < 1/L$ and $\alpha_{\lambda} < m/\overline{\sigma}^2$, they establish the decrease condition $V_{k+1} \leq \rho^2\,V_k$ with
\[
    \rho^2 = \max\{1-\alpha_{x} m\,(1-\alpha_{x}L),\,1-\alpha_{x}\alpha_{\lambda}\underline{\sigma}^2\}.
\]
The bound on the convergence rate is optimized by the step sizes
\begin{subequations}\label{stepsizes2}
\begin{align}
    \alpha_x &= \begin{cases} \frac{1}{2L} & \text{if }\kappa(A)\leq\sqrt{2} \\ \frac{1-\kappa(A)^{-2}}{L} & \text{otherwise} \end{cases}\\
    \alpha_\lambda &= \begin{cases} \frac{m}{4}\bigl(\frac{2}{\overline{\sigma}^2} + \frac{1}{\underline{\sigma}^2}\bigr) & \text{if }\kappa(A)\leq\sqrt{2} \\ \frac{m}{\overline{\sigma}^2} & \text{otherwise}, \end{cases}
\end{align}
\end{subequations}
and the convergence factor using these step sizes is
\begin{equation}\label{bound2}
    \rho^2 = \begin{cases} 1-\frac{1}{4\,\kappa(f)} & \text{if }\kappa(A)\leq\sqrt{2} \\ 1 - \frac{1}{\kappa(f)} \bigl(\frac{1}{\kappa(A)^2} - \frac{1}{\kappa(A)^4}\bigr) & \text{otherwise}. \end{cases}
\end{equation}

% \subsection{The Case of Uniform Quadratic Objectives}
% \label{Sec:Uniform}

% Further insight can be gained by examining the case of \emph{uniform} quadratic objective functions
% \[
% f(x) = \tfrac{1}{2}q x^2, \qquad q > 0.
% \]
% In this case, \eqref{Eq:AHU-Extended} becomes a linear time-invariant dynamical system with system matrix
% \[
% M = \begin{bmatrix}
% (1-\alpha_x q)I_n - \alpha_x \mu A^{\tp}A & - \alpha_x A^{\tp}\\
% \alpha_{\lambda}(I_n-\alpha_x\gamma(q I + \mu AA^{\tp}))A & 1-\alpha_{\lambda}\alpha_{x}AA^{\tp}
% \end{bmatrix}
% \]
% In this case, the asymptotic convergence rate of \eqref{Eq:AHU-Extended} is governed by the spectral radius $\rho(M)$ of $M$. Let
% \[
% A = U \begin{bmatrix}\Sigma & 0\end{bmatrix}V^{\tp}
% \]
% denote the singular value decomposition of $A$, and set $T = \mathrm{diag}(V^{\tp},U^{\tp})$. One quickly computes then that
% \[
% T^{-1}MT = \begin{bmatrix}
% (1-\alpha_{x}q)I_r-\alpha_{x}\mu \Sigma^2 & 0 & -\alpha_x \Sigma\\
% 0 & (1-\alpha_{x}q)I_{n-r} & 0\\
% \alpha_{\lambda}\Sigma - \alpha_{x}\alpha_{\lambda}\gamma(q
% \end{bmatrix}
% \]

\section{Automated convergence analysis}\label{Sec:analysis}

%Let $V_1 \in \real^{n \times r}$ and $V_2 \in \real^{n \times (n-r)}$ be matrices with orthonormal columns such that $\mathrm{Im}(V_1) = \mathrm{Im}(A^{\tp})$ and $\mathrm{Im}(V_2) = \mathrm{Ker}(A)$, and consider the invertible change of variables defined by $p_k = V_1^{\tp}x_k$ and $q_k = V_2^{\tp}$

We now return to the primal-dual algorithm \eqref{Eq:AHU-Extended} and rewrite it in the form of a linear fractional representation, as traditionally used in robust control \cite{CS-SW:15}. Let $A = U\Sigma V_1^{\tp}$ be a compact SVD of $A$, where $U\in\real^{r\times r}$ is orthogonal and $V_1 \in \real^{n\times r}$ has orthonormal columns. Let $V_2\in \real^{n\times (n-r)}$ be the orthogonal completion, so that $V = \bmat{V_1 & V_2} \in \real^{n\times n}$ is orthogonal. Here, $\Sigma = \mathrm{diag}(\sigma_1,\dots,\sigma_r)$ and we have the inequality $\bar\sigma \geq \sigma_1 \geq \dots \geq \sigma_r \geq \underline\sigma$. Consider now the invertible change of state
\[
p_k = V_1^{\tp}x_k, \qquad q_k = V_2^{\tp}x_k, \qquad \nu_{k} = -\Sigma U^{\tp}\lambda_k.
\]
Routine computations quickly show that
\[
\begin{aligned}
\begin{bmatrix}p_{k+1} \\ q_{k+1}\end{bmatrix} &= \begin{bmatrix}p_{k} \\ q_{k}\end{bmatrix} - \alpha_{x}V^{\tp}\nabla f(V \left[\begin{smallmatrix}p_k \\ q_k\end{smallmatrix}\right]) + \alpha_x \begin{bmatrix} \nu_k - \mu \Sigma^2 p_k \\ 0\end{bmatrix}\\
\tilde{p}_{k} &= p_{k} + \gamma\,(p_{k+1} - p_{k})\\
\nu_{k+1} &= \nu_{k} - \alpha_{\lambda} \Sigma^2 \tilde{p}_k.
\end{aligned}
\]
Eliminating $\tilde{p}_k$ and defining the concatenated state $\xi_k = (p_k,q_k,\nu_k)$, the dynamics can be expressed as
\[
\begin{aligned}
\xi_{k+1} &= A\xi_k + B_1 \left[\begin{smallmatrix}u_{k}^{1} \\ u_{k}^{2}\end{smallmatrix}\right] + B_2 \left[\begin{smallmatrix}u_{k}^{3} \\ u_{k}^{4}\end{smallmatrix}\right]\\
\left[\begin{smallmatrix}y_{k}^{1} \\ y_{k}^{2}\end{smallmatrix}\right] &= C_1 \xi_k + D_{11}\left[\begin{smallmatrix}u_{k}^{1} \\ u_{k}^{2}\end{smallmatrix}\right] + D_{12}\left[\begin{smallmatrix}u_{k}^{3} \\ u_{k}^{4}\end{smallmatrix}\right]\\
\left[\begin{smallmatrix}y_{k}^{3} \\ y_{k}^{4}\end{smallmatrix}\right] &= C_2 \xi_k + D_{21}\left[\begin{smallmatrix}u_{k}^{1} \\ u_{k}^{2}\end{smallmatrix}\right] + D_{22}\left[\begin{smallmatrix}u_{k}^{3} \\ u_{k}^{4}\end{smallmatrix}\right]
\end{aligned}
\]
where the matrices are defined by the blocks
\[
{\small\addtolength{\arraycolsep}{-1.2mm}\bmat{A & B_1 & B_2 \\ C_1 & D_{11} & D_{12} \\ C_2 & D_{21} & D_{22}} \!=\!
\addtolength{\arraycolsep}{0mm}
\left[\begin{array}{ccc|cc|cc}
I & 0 & \alpha_x I & -\alpha_{x} I & 0     & -\alpha_{x} \mu I & 0\\
0 & I & 0          & 0 & -\alpha_{x} I     & 0 & 0\\
0 & 0 & I          & 0 & 0                 & 0 & -\alpha_{\lambda}I\\ 
\hline
    I & 0  & 0         & 0 & 0                 & 0 & 0\\
0 & I & 0          & 0 & 0                 & 0 & 0\\
\hline
I & 0 & 0          & 0 & 0                 & 0 & 0\\
I & 0 & \gamma\alpha_{x}I       & -\gamma\alpha_{x} I & 0       & -\gamma\alpha_x\mu I & 0
\end{array}\right]
}
\]
and with inputs $(u_{k}^{1},u_{k}^{2})$ and $(u_{k}^{3},u_{k}^{4})$ defined by
\begin{equation}\label{eq:feedback}
\left[\begin{smallmatrix}u_{k}^{1} \\ u_{k}^{2}\end{smallmatrix}\right] = V^{\tp}\nabla f\Bigl(V\left[\begin{smallmatrix}y_{k}^{1} \\ y_{k}^{2}\end{smallmatrix}\right]\Bigr), \quad \left[\begin{smallmatrix}u_{k}^{3} \\ u_{k}^{4}\end{smallmatrix}\right] = \left[\begin{smallmatrix}\Sigma^2 & 0\\
0 & \Sigma^2\end{smallmatrix}\right]\left[\begin{smallmatrix}y_{k}^{3} \\ y_{k}^{4}\end{smallmatrix}\right]
\end{equation}
and outputs $(y_k^1,y_k^2,y_k^3,y_k^4) = (p_k,q_k,p_k,\tilde p_k)$. To simplify notation in the sequel, we assume without loss of generality that each input and output is one-dimensional; see the lossless dimensionality reduction in \cite{lessard}.

\subsection{Lifted dynamics}

Let $G(z)$ denote the transfer function corresponding to the state-space matrices above, which maps the set of inputs $(u_k^1,u_k^2,u_k^3,u_k^4)$ to the set of outputs $(y_k^1,y_k^2,y_k^3,y_k^4)$. This system is connected with the feedback in \eqref{eq:feedback} through the gradient of the objective function and the squared matrix of singular values.

To analyze the system, we will replace the feedback in \eqref{eq:feedback} with constraints on the inputs and outputs of $\nabla f$ and $\Sigma^2$. The more constraints that we use, the tighter the analysis will be. To obtain more constraints, we will lift the system so that the inputs and outputs are in a higher-dimensional space, and then apply the constraints between all iterates in this lifted space \cite{lifting}.

Given a lifting dimension $\ell\in\{1,2,\ldots\}$, let
\begin{equation}\label{Eq:LiftedSystem}
    \psi(z) = I_4\otimes\bmat{1 \\ z^{-1} \\ \vdots \\ z^{1-\ell}}
    \quad\text{and}\quad
    \GG = \bmat{\psi G \\ \psi}
\end{equation}
where $\otimes$ denotes the Kronecker product. The system $\GG$, called the \textit{lifted system}, is the map
\begin{subequations}\label{eq:lifted-iterates}
\begin{equation}\label{eq:lifted-map}
    (u_k^1,u_k^2,u_k^3,u_k^4) \mapsto (Y_k^1,\ldots,Y_k^4,U_k^1,\ldots,U_k^4)
\end{equation}
where the lifted iterates are
\begin{equation}
    Y_k^i = (y_k^i,y_{k-1}^i,\ldots,y_{k-\ell+1}^i), \quad i\in\{1,2,3,4\}
\end{equation}
\end{subequations}
and similarly for $U_k^i$. Each lifted iterate $Y_k^i$ and $U_k^i$ is an $\ell$-dimensional vector that consists of $\ell$ lagged iterates, where the lifting dimension is how many past iterates are used. When $\ell=1$, the lifted system is simply $\GG = \left[\begin{smallmatrix} G \\ I \end{smallmatrix}\right]$.

\subsection{Multipliers}

To analyze the system, we will replace the feedback \eqref{eq:feedback} with inequalities on the inputs and outputs of the system in the lifted space. We parameterize the set of inequalities using a symmetric block matrix, called a \textit{multiplier}, of the form
\begin{equation}\label{multiplier}
    M = \bmat{M_{11} & M_{12} \\ M_{21} & M_{22}}.
\end{equation}
We now describe the constraints for both the objective function and constraint matrix.

\paragraph{Objective function.}

We consider inequalities on the objective function gradient of the form
\begin{multline}\label{eq:F-quadratic-inequality}
  0 \leq \sum_{i=1}^\ell \sum_{j=1}^\ell \biggl((M_{11})_{ij}\,\bmat{y_{k-i+1}^1 \\ y_{k-i+1}^2}^\tp \bmat{y_{k-j+1}^1 \\ y_{k-j+1}^2} \\
  + 2\,(M_{12})_{ij}\,\bmat{y_{k-i+1}^1 \\ y_{k-i+1}^2}^\tp \bmat{u_{k-j+1}^1 \\ u_{k-j+1}^2} \\
  + (M_{22})_{ij}\,\bmat{u_{k-i+1}^1 \\ u_{k-i+1}^2}^\tp \bmat{u_{k-j+1}^1 \\ u_{k-j+1}^2}\biggr).
\end{multline}
The multiplier $M$ has dimensions $2\ell\times 2\ell$ and parameterizes all inequalities that are linear in the inner products between inputs and outputs of the gradient. The following result characterizes all multipliers such that this inequality holds for all iterates of the system. We provide the proof in the appendix.

\begin{proposition}[Objective multipliers]\label{prop:objective-multipliers}
    The quadratic inequality \eqref{eq:F-quadratic-inequality} holds for all iterates that satisfy the feedback \eqref{eq:feedback} for some function $f\in\F(m,L)$ if and only if the multiplier has the form
    \begin{equation}\label{Eq:ObjectiveMult}
        M = \bmat{-2mL & L+m \\ L+m & -2} \otimes R + \bmat{\mathbf{0} & S \\ S^\tp & \mathbf{0}}
    \end{equation}
    for some $\ell\times\ell$ symmetric matrix $R$ such that $R\mathbf{1}=0$, $R_{ii}\geq 0$ for all $i$, and $R_{ij}\leq 0$ for all $i\neq j,$\footnote{Matrices $R$ satisfying these conditions are also called \emph{diagonally hyperdominant with zero excess}.} and some $\ell\times\ell$ skew-symmetric matrix $S$ satisfying $S\mathbf{1} = \mathbf{0}$. We denote the set of all such matrices as $\M_\F(m,L)$.
\end{proposition}

\paragraph{Constraint matrix.}

We consider inequalities on the constraint matrix of the form
\begin{equation}\label{eq:A-quadratic-inequality}
  0 \leq \trace\left(M \bmat{Y_k^3 \\ Y_k^4 \\ U_k^3 \\ U_k^4} \bmat{Y_k^3 \\ Y_k^4 \\ U_k^3 \\ U_k^4}^\tp\right).
\end{equation}
The multiplier $M$ has dimensions $4\ell\times4\ell$ and parameterizes all inequalities that are linear in the inner products between inputs and outputs of the matrix $\Sigma^2$ of squared singular values of $A$. The following result characterizes the set of multipliers, which we prove in the appendix.

\begin{proposition}[Constraint multipliers]\label{prop:constraint-multipliers}
    The quadratic inequality \eqref{eq:F-quadratic-inequality} holds for all iterates that satisfy the feedback \eqref{eq:feedback} for some matrix $A\in\A(\underline{\sigma},\overline{\sigma})$ if the multiplier has the form
    \begin{equation}\label{Eq:ObjectiveConst}
        M = \bmat{-2\,\underline{\sigma}^2\,\overline{\sigma}^2 R & (\overline{\sigma}^2+\underline{\sigma}^2) R \\ (\overline{\sigma}^2+\underline{\sigma}^2) R & -2 R} + \bmat{0 & S \\ S^\tp & 0}
    \end{equation}
    for some $2\ell\times2\ell$ symmetric positive semidefinite matrix $R$, and some $2\ell\times 2\ell$ skew-symmetric matrix $S$. We denote the set of all such matrices as $\M_\A(\underline{\sigma},\overline{\sigma})$.
\end{proposition}

\subsection{Linear matrix inequality}

We now use the lifted system \eqref{Eq:LiftedSystem} and the multipliers \eqref{Eq:ObjectiveMult} and \eqref{Eq:ObjectiveConst} that characterize the objective function and constraint matrix to construct a linear matrix inequality (LMI) whose feasibility certifies convergence of the primal-dual algorithm \eqref{Eq:AHU-Extended} with a specified rate $\rho \in (0,1)$. 

Denote a minimal realization of the lifted system as $\GG \sim (\AA,\BB,\CC,\DD)$ and let $\nn$ denote the dimension of the realization. Recall that the lifted system maps the iterates as in \eqref{eq:lifted-iterates}. Let $\YF$ and $\UF$ denote the rows of $\CC$ and $\DD$ corresponding to pairs of inputs and outputs of the gradient, and let $\YA$ and $\UA$ denote rows corresponding to inputs and outputs of $\Sigma^2$. 
% Also, define the $1\times\nn$ vector
% \[
%     X = \bmat{1 & 0 & \ldots & 0}.
% \]
We can now state our main result.

\begin{theorem}[Analysis]\label{thm:analysis}
    Given $\rho\in(0,1)$, if there exists an $\nn\times\nn$ symmetric matrix $P$ and multipliers $M_1,M_2\in\M_\F(m,L)$ and $M_3,M_4\in\M_\A(\underline{\sigma},\overline{\sigma})$ such that
    \begin{subequations}\label{lmi}
    \begin{multline}\label{lmi1}
        0 \succeq \bmat{\AA^\tp P\AA - \rho^2 P & \AA^\tp P \BB \\ \BB^\tp P \AA & \BB^\tp P \BB} + \bmat{\YF \\ \UF}^\tp\! M_1 \bmat{\YF \\ \UF} \\
        + \bmat{\YA \\ \UA}^\tp\! M_3 \bmat{\YA \\ \UA}
    \end{multline}
    and
    \begin{multline}\label{lmi2}
        0 \preceq \bmat{P-I & 0 \\ 0 & 0} + \bmat{\YF \\ \UF}^\tp\! M_2 \bmat{\YF \\ \UF} \\
        + \bmat{\YA \\ \UA}^\tp\! M_4 \bmat{\YA \\ \UA},
    \end{multline}
    \end{subequations}
    then the primal-dual iterations from \eqref{Eq:AHU-Extended} converge linearly with rate $O(\rho^k)$ for all objective functions $f\in\F(m,L)$ and all constraint matrices $A\in\A(\underline{\sigma},\overline{\sigma})$.
\end{theorem}

\begin{proof}
    Suppose the LMI \eqref{lmi} is feasible, and consider a trajectory of the primal-dual algorithm \eqref{Eq:AHU-Extended}. Let $\bm{\xi}_k$ denote the state of the lifted system $\GG$. For each iterate, let a tilde denote the iterate shifted by the fixed point of the system. Now multiply the LMI in \eqref{lmi1} on the right and left by the lifted state and inputs $(\bm{\tilde\xi}_k,\tilde u_k^1,\tilde u_k^2,\tilde u_k^3,\tilde u_k^4)$ and its transpose and use the fact that the inequalities \eqref{eq:F-quadratic-inequality} and \eqref{eq:A-quadratic-inequality} are nonnegative when $M_1,M_2\in\M_\F(m,L)$ and $M_3,M_4\in\M_\A(\underline{\sigma},\overline{\sigma})$. This produces the inequality $V(\bm{\tilde\xi}_{k+1}) \leq \rho^2\,V(\bm{\tilde\xi}_k)$. Likewise, from the LMI \eqref{lmi2}, we obtain the inequality $V(\bm{\tilde\xi}_k)\geq \|\bm{\tilde\xi}_k\|^2$. Since the state $\tilde\xi_k$ of the original system $G$ is contained in the lifted system, we have that $\|\bm{\tilde \xi}_k\|^2 \geq \|\tilde\xi_k\|^2$. Chaining all of these inequalities together gives
    \[
        \|\tilde\xi_k\|^2 \leq \|\bm{\tilde\xi}_k\|^2 \leq V(\bm{\tilde\xi}_k) \leq \ldots \leq \rho^{2k}\,V(\bm{\tilde\xi}_0).
    \]
    Taking the square root gives $\|\tilde\xi_k\|\leq c\,\rho^k$ for some $c>0$, so the iterates converge to the optimizer linearly with rate $O(\rho^k)$.
\end{proof}

\section{Results}\label{Sec:results}

We now compare our analysis with the results from the literature described in Section~\ref{Sec:literature}. In each case, we choose the objective function condition number $\kappa(f) = 2$ and the Lagrangian augmentation parameter $\mu=0$. We emphasize that our analysis applies to any values of these parameters, but we choose these values to be able to compare with known bounds.

\begin{figure}[ht]
    \includegraphics{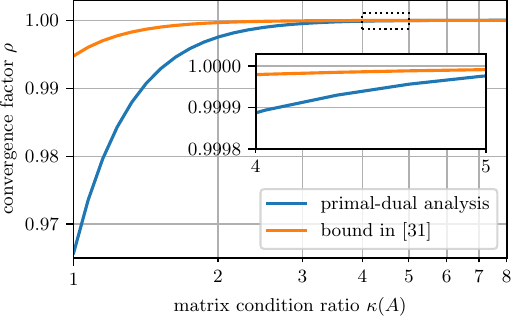}
    \caption{Comparison with \cite{SSD-WH:19}. The algorithm parameters are those in \eqref{stepsizes1} with extrapolation parameter $\gamma=0$.}
    \label{fig:comparison1}
\end{figure}

We use the step size selections from \eqref{stepsizes1} and \eqref{stepsizes2}; Figures \ref{fig:comparison1} and \ref{fig:comparison2} plot the convergence factor $\rho$ obtained\footnote{A bisection search was applied to determine the smallest $\rho$ for which the LMIs were feasible.} from our analysis in Theorem \ref{thm:analysis} along with the corresponding bound from the literature as a function of the matrix condition number $\kappa(A)$. In all cases, our approach certifies a faster rate of convergence. Note that the algorithm in Figure \ref{fig:comparison1} does not use extrapolation, while the algorithm in Figure \ref{fig:comparison2} does use extrapolation, and achieves a much faster convergence rate as a result.

\begin{figure}[ht]
    \centering
    \includegraphics{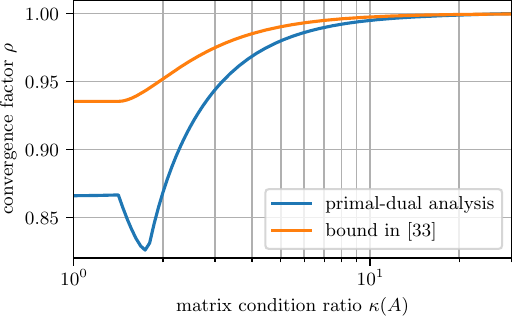}
    \caption{Comparison with \cite{SAA-AHS:20}. The algorithm parameters are those in \eqref{stepsizes2} with extrapolation parameter $\gamma=1$.}
    \label{fig:comparison2}
\end{figure}

The guaranteed convergence factor produced by our approach with step size selection \eqref{stepsizes2} is a non-monotonic function of $\kappa(A)$ in Figure \ref{fig:comparison2}. For any \textit{fixed} algorithm, the convergence rate must increase monotonically with the condition numbers $\kappa(f)$ and $\kappa(A)$. However, the step size selection \eqref{stepsizes2} is a function of $\kappa(A)$, and thus the algorithm used is varying at each point on this curve. This highlights that the bounds from the literature are conservative, and suggests that further improvements in the worst-case convergence rate could be obtained by optimizing the step-sizes subject to feasibility of our LMI.

\section{Conclusion}\label{Sec:conclusion}

Focusing on first-order primal-dual methods for linearly constrained convex optimization, we proposed a systematic method to search for a Lyapunov function that certifies a worst-case convergence rate of the algorithm across all smooth strongly convex functions and all constraint matrices with bounded singular values. Our analysis applies to a range of primal-dual algorithms, including those with extrapolation and Lagrangian augmentation. We compared our numerical results with two bounds from the literature, and our analysis yields better bounds in each case. Future work includes finding algorithm parameters that optimize the rates obtained from our analysis, improving the analysis by leveraging recent interpolation conditions for linear operators \cite{linear-interpolation}, and extending the approach to more general operator splitting methods \cite{LC-DK-AC-AK:23}.

%The analysis is tight in that, if the LMI is infeasible, then there does not exist a Lyapunov function of the proposed form (although there may exist a more complex Lyapunov function).

%%%%%%%%%%%%%%%%%%%%%%%%%%%%%%%%%%%%%%%%%%%%%%%%%%%%%%%%%%%%%%%%%%%%%%%%%%
%\bibliographystyle{abbrv}
%{\small \bibliography{references}}

\bibliographystyle{IEEEtran}

{\footnotesize
\bibliography{brevalias, Main2, JWSP, references}
}

% \color{blue}

\appendix

\section{Proof of Proposition \ref{prop:objective-multipliers}}

Let $y_i,u_i,f_i$ for $i\in [\ell]=\{1,\ldots,\ell\}$ denote a sequence of $\ell$ iterates. From \cite{interpolation}, these points are interpolable by an $L$-smooth and $m$-strongly convex function if and only if
\[
    q_{ij} := \sbmat{
    y_i \\ y_j \\ u_i \\ u_j}^\tp\! \underbrace{\sbmat{
    -mL & mL & m & -L\\
    mL & -mL & -m & L\\
    m & -m & -1 & 1\\
    -L & L & 1 & -1
    }}_H \sbmat{
    y_i \\ y_j \\ u_i \\ u_j} + h^\tp \sbmat{f_i \\ f_j} \geq 0
\]
for all $i,j\in [\ell]$ where $h = 2(L-m)\sbmat{1 \\ -1}$. In terms of the stacked vectors $u,y,f$,
\[
    q_{ij} = \bmat{y \\ u}^\tp\! H_{ij} \bmat{u \\ u} + h_{ij}^\tp f
\]
where $h_{ij} = \sbmat{e_i & e_j} h = 2(L-m)(e_i-e_j)$ and $H_{ij}$ is the matrix
\begin{align*}
    \sbmat{e_i^\tp & 0_\ell^\tp \\ e_j^{\tp} & 0_{\ell}^\tp \\
    0_{\ell}^{\tp} & e_i^\tp \\ 0_{\ell}^{\tp} & e_j^\tp}^\tp\!\!\!\!\!
    H\!
    \sbmat{e_i^\tp & 0_\ell^\tp \\ e_j^{\tp} & 0_{\ell}^\tp \\
    0_{\ell}^{\tp} & e_i^\tp \\ 0_{\ell}^{\tp} & e_j^\tp}
    \!\!=\!\! \sbmat{
    \!\!-mL(e_i-e_j)(e_i-e_j)^\tp & (e_i-e_j)(me_i - Le_j)^\tp\!\\
    (me_i - Le_j)(e_i-e_j)^\tp & -(e_i-e_j)(e_i-e_j)^\tp
    }
\end{align*}
where $e_i$ is the $i\textsuperscript{th}$ unit vector in $\real^\ell$. Taking a nonnegative linear combination of the inequalities $q_{ij} \geq 0$, we obtain $Q = \sum_{i \neq j}\lambda_{ij}\,q_{ij} \geq 0$ for all coefficients $\lambda_{ij} \geq 0$ with $\lambda_{ii} = 0$ for all $i$. With $\Lambda \in \real^{\ell \times \ell}_{\geq 0}$ the matrix with elements $\lambda_{ij}$, straightforward but tedious algebra establishes that
\begin{multline*}
    Q = \bmat{y \\ u}^\tp
    \bmat{ (-mL) R & \tfrac{m+L}{2}R + \tfrac{m-L}{2}T\\ 
    \star & -R}
    \bmat{y \\ u}\\
    +(L-m)\mathbf{1}^{\tp}(T-T^{\tp})f,
\end{multline*}
where the $\star$ block can be inferred from symmetry and
\begin{align*}
    R &\defeq \diag((\Lambda+\Lambda^\tp) \mathbf{1})-(\Lambda+\Lambda^\tp), \\
    T &\defeq \diag((\Lambda-\Lambda^\tp) \mathbf{1}) + (\Lambda-\Lambda^\tp).
\end{align*}
Note that $R$ is symmetric doubly hyperdominant with zero row sums, and $R$ and $T$ are independent, as they depend on the symmetric and skew-symmetric parts of $\Lambda$, respectively. Now define the skew symmetric matrix $S \defeq \tfrac{m-L}{2}(\Lambda - \Lambda^{\tp})$ and note that $\tfrac{m-L}{2}T = \mathrm{diag}(S\mathbf{1}) + S$ and $\tfrac{m-L}{2}T^{\tp} = \mathrm{diag}(S\mathbf{1}) - S = \mathrm{diag}(S\mathbf{1}) + S^{\tp}$. 
% We may now write $Q = \left[\begin{smallmatrix}y \\ u\end{smallmatrix}\right]^{\tp}M\left[\begin{smallmatrix}y \\ u\end{smallmatrix}\right] + \eta^{\tp}f$ where
% \[
% \begin{aligned}
% M &= \sbmat{-2mL & m+L\\ m+L & -2}\otimes R + \sbmat{\boldsymbol{0} & \mathrm{diag}(S\mathbf{1})+S\\ \mathrm{diag}(S\mathbf{1})+S^{\tp} & \boldsymbol{0}},\\
% \eta &= -4(S\mathbf{1}).
% \end{aligned}
% \]
Then the nonnegative quantity $Q$ is
\[
    \sbmat{y \\ u}^\tp\bigl(\sbmat{-2mL & m+L\\ m+L & -2}\otimes R + \sbmat{\boldsymbol{0} & \mathrm{diag}(S\mathbf{1})+S\\ \star & \boldsymbol{0}}\bigr) \sbmat{y \\ u} - 4(S\mathbf{1})^\tp f,
\]
which is of the form \eqref{eq:F-quadratic-inequality} if and only if $S\mathbf{1} = \mathbf{0}$.\qedhere

\section{Proof of Proposition \ref{prop:constraint-multipliers}}

Consider a matrix $M$ of this form, and define the matrices
\[
    M_0 = \bmat{-2\,\underline{\sigma}^2\,\overline{\sigma}^2 & \overline{\sigma}^2+\underline{\sigma}^2 \\ \overline{\sigma}^2+\underline{\sigma}^2 & -2}
    \quad\text{and}\quad
    M_1 = \bmat{0 & 1 \\ -1 & 0}
\]
so that the multiplier is $M = M_0\otimes R + M_1\otimes S$. We first write the inequality \eqref{eq:A-quadratic-inequality} as
\[
    0\leq\sbmat{Y \\ U}^\tp\! M \sbmat{Y \\ U} \quad\text{where}\quad
    Y = \sbmat{Y_k^3 \\ Y_k^4}, \quad
    U = \sbmat{U_k^3 \\ U_k^4}.
\]
From the feedback \eqref{eq:feedback}, we have that $U = (I_{2\ell}\otimes\Sigma^2) Y$, so the matrix $YU^\tp = UY^\tp$ is symmetric. For the term $M_1\otimes S$ of the multiplier, the quadratic form is
\begin{align*}
    \sbmat{Y \\ U}^\tp\! (M_1\otimes S) \sbmat{Y \\ U}
    &= Y^\tp S U - U^\tp S Y \\
    &= \trace\bigl( S( UY^\tp -YU^\tp )\bigr)
    = 0.
\end{align*}
Therefore, this term does not affect the inequality. Without loss of generality, we can take $M_1\otimes S$ to be symmetric, in which case $S$ is skew-symmetric. For the term $M_0\otimes R$ of the multiplier, the quadratic form is
\begin{multline*}
    \sbmat{Y \\ U}^\tp (M_0\otimes R) \sbmat{Y \\ U}
    = (U-\underline{\sigma}^2 Y)^\tp R\,(\overline{\sigma}^2 Y - U) \\
    = \bigl((I\otimes\Sigma^2)-\underline{\sigma}^2 I\bigr) Y^\tp R Y \bigl(\overline{\sigma}^2-(I\otimes\Sigma^2)\bigr).
\end{multline*}
This quantity is nonnegative since $\Sigma$ is a diagonal matrix of singular values in the interval $[\underline{\sigma},\overline{\sigma}]$ and $Y^\tp R Y$ is positive semidefinite. Therefore, the inequality \eqref{eq:A-quadratic-inequality} holds for iterates that satisfy \eqref{eq:feedback} for all $M \in \M_\A(\underline{\sigma},\overline{\sigma})$.\qedhere

\end{document}